\documentclass[11pt]{amsart}
\usepackage{amssymb}
\usepackage{times}
\usepackage{url}
\usepackage{hyperref}
\usepackage{pdflscape,longtable}

\newtheorem{theorem}{Theorem}[section]
\newtheorem{lemma}[theorem]{Lemma}
\newtheorem{proposition}[theorem]{Proposition}

\theoremstyle{remark}

\theoremstyle{example}

\newcommand{\F}{\mathbb{F}}
\newcommand{\Z}{\mathbb{Z}}
\newcommand{\Q}{\mathbb{Q}}
\newcommand{\C}{\mathbb{C}}

\newcommand{\abk}{\allowbreak} 

\newcommand{\GL}{\mathrm{GL}}
\newcommand{\SL}{\mathrm{SL}}

\newcommand{\Sp}{\mathrm{Sp}}

\newcommand{\cl}{\mathrm{cl}}

\begin{document}

\title[Experimenting with symplectic hypergeometric monodromy
groups]{Experimenting with symplectic hypergeometric monodromy
groups}

\author{A.~S.~Detinko}
\address{Department of Physics and Mathematics\\
Faculty of Science and Engineering\\
University of Hull\\
Hull HU6 7RX\\
UK}
\email{a.detinko@hull.ac.uk}

\author{D.~L.~Flannery}
\address{School of Mathematics, Statistics 
 and Applied Mathematics\\
National University of Ireland, Galway\\
Galway H91TK33\\
Ireland}
\email{dane.flannery@nuigalway.ie}

\author{A.~Hulpke}
\address{Department of Mathematics\\
Colorado State University\\
Fort Collins\\
 CO 80523-1874\\
USA}
\email{Alexander.Hulpke@colostate.edu}

\subjclass[2010]{20-04, 20G15, 20H25, 68W30}
\keywords{symplectic group, Zariski density, 
strong approximation, 
algorithm}

\begin{abstract}
We present new computational results for symplectic 
monodromy groups of hypergeometric differential
equations. 
In particular, we compute the arithmetic closure of 
each group, sometimes justifying arithmeticity. 
The results are obtained by extending
our earlier algorithms for Zariski dense groups, 
based on the strong approximation and congruence 
subgroup properties.
\end{abstract}

\maketitle

\section{Introduction}~\label{intro}

This paper is a sequel to
\cite{Density}, which treated symplectic monodromy 
groups of hypergeometric differential equations as a 
test case.
Deciding arithmeticity of such a group   
in its Zariski closure is a basic problem
(see \cite[Section~3.5]{Sarnak} and \cite[p.~326]{BH}). 
More generally, one asks whether the group is arithmetic, 
or whether it is \emph{thin}, i.e., Zariski dense but not 
arithmetic in the ambient algebraic group.
This problem has received much attention.
It was solved completely for monodromy 
groups associated with Calabi--Yau 
manifolds~\cite{BravThomas,Singh,SinghVenky}, which
are $4$-dimensional symplectic linear groups
over $\Q$. Note also the results of \cite{FuchsMonodromy}, 
demonstrating thinness of certain orthogonal 
hypergeometric monodromy groups. 

Our approach to all questions emphasizes computer-aided 
experimentation. We compute the 
arithmetic closure $\cl(H)$ of a dense group $H$, 
the `closest' arithmetic overgroup of $H$ (specifically,
we compute the index of $\cl(H)$, and the level of the 
maximal principal congruence subgroup that $\cl(H)$ 
contains). Then $\cl(H)$ is used to investigate $H$.
Sometimes we are able to prove that $H$ is arithmetic.
Moreover, we produce comprehensive extra information 
about all symplectic hypergeometric monodromy groups in the
degrees considered.

Our methods are based on the strong approximation 
and congruence subgroup properties for the symplectic group.   
In Section~\ref{SATiCSP}, we extend algorithms developed 
in \cite{Density} for dense subgroups of $\Sp(n, \Z)$ to 
accept dense subgroups of $\Sp(n, \Q)$.  
Section~\ref{Hyper} provides relevant background on 
hypergeometric groups, and details of our experimental 
strategy.
Output for all dense hypergeometric 
monodromy subgroups of the symplectic group of degree $4$
is tabulated in Section~\ref{append}. We discuss
this data in light of work by other authors, noting  
new proofs of arithmeticity and new index calculations.
As further illustration, we give sample output for 
groups of degree $6$.

We set down some notation. Input groups for all algorithms
are finitely generated. Let $H = \langle S\rangle$ 
where $S = \{g_1, \ldots , g_r\}
\subseteq \GL(n,\Q)$. The subring of
$\Q$ generated by the entries of the $g_i$ and $g_i^{-1}$ 
will be denoted $R$.
Thus $R=\frac{1}{\mu}\Z$ for a 
positive integer $\mu$.
If $m$ is coprime to $\mu$ then the congruence 
homomorphism $\varphi_m$ induced by natural surjection 
$\Z\rightarrow \Z_m = \Z/m\Z$ maps $\GL(n, R)$ into 
$\GL(n, \Z_m)$. 

Throughout, $\F$ is a field and $1_m$ is the $m\times m$ 
identity matrix. Let $V$ be the $\F$-vector 
space of dimension $n=2s>2$, and let $\Phi$ be the 
matrix of a non-degenerate skew-symmetric bilinear 
form on $V$ with respect to a basis of $V$. 
The full symplectic group in $\GL(n,\F)$ preserving 
$\Phi$ is denoted $\Sp(\Phi, \F)$. 
If $D\subseteq\F$ is a unital subring then 
$\Sp(\Phi,D):=  \Sp(\Phi, \F)\cap \GL(n,D)$.
We write $\Sp(n, D)$ instead of $\Sp(\Phi, D)$ if
\[
\Phi = J_n:={\small \left(
\!\!
\renewcommand{\arraycolsep}{.15cm}
\begin{array}{rr} 0_{s} & 1_{s}
\\
\vspace*{-9pt}& \\
-1_{s} & 0_{s}
\end{array} \! \right)}.
\]   
Since $\Sp(\Phi, \F)$ and $\Sp(n, \F)$ are 
$\GL(n,\F)$-conjugate, often it suffices to deal with 
the latter rather than the former group.
The shorthand $\Sp_n$ stands for the symplectic 
group when $\F$ and $\Phi$ are unimportant.

\section{Computing with dense subgroups of symplectic groups}
\label{SATiCSP}

In this section we establish the theoretical foundation for
our algorithms. 

\subsection{Strong approximation and computing}
Let $H$ be a finitely generated dense subgroup 
of $\Sp(n, \Q)$.
The strong approximation theorem guarantees 
that $H$ surjects onto $\Sp(n, p)$ 
for almost all primes 
$p\in \Z$~\cite[Corollary~3, Window~9]{LubotzkySegal}.
Let $\Pi(H)$ be the (finite) set of primes $p$ such that 
$p\nmid \mu$ and $\varphi_p(H) \neq \Sp(n, p)$. 
Below we outline how to compute $\Pi(H)$. 

In \cite{SAT} we developed a method to compute the set
 of primes $p$ such that $\varphi_p(H) \neq \abk \SL(n, p)$ 
for dense $H\leq \SL(n, \Q)$. This relies on
testing irreducibility of the adjoint module of $H$, 
and the classification of maximal subgroups of 
$\SL(n, p)$. Something similar could be done
for dense subgroups of $\Sp(n, \Q)$. 

For a dense subgroup $H$ of $\SL(n, \Z)$ or 
$\Sp(n, \Z)$, another way to compute $\Pi(H)$ is 
described in \cite[Section~3]{Density} (see also 
\cite[Section~2.5]{DensityFurther}).
Here we must know an explicit transvection in $H$ 
(recall that a transvection $\tau\in \GL(n,\F)$ is a 
unipotent element such that $1_n-\tau$ has 
rank $1$). While an arbitrary dense subgroup of 
$\Sp_n$ may not contain a transvection, the groups in 
our experiments do.
\begin{proposition}
\label{dense}
Suppose that $H$ is a finitely generated 
subgroup of $\Sp(n, \Q)$ containing a 
transvection $\tau$. Then $H$ is dense if
and only if $\langle \tau \rangle^H$ is absolutely 
irreducible.
\end{proposition}
\begin{proof}
The proof of \cite[Proposition~3.7]{Density} for 
$H \leq \Sp(n, \Z)$ remains valid for $H\leq \Sp(n,R)$.
\end{proof}

Having identified a transvection 
$\tau$ in $H\leq \Sp(n, \Q)$, 
Proposition~\ref{dense} allows us to apply the 
procedure ${\tt IsDense}(H, \tau)$ from 
\cite[Section~3.2]{Density} to test density of 
$H$.  Given dense $H$, we compute $\Pi(H)$ 
using ${\tt PrimesForDense}(H, \tau)$ from 
\cite[Section~3.2]{Density} as follows.
Let $\{A_1 , \ldots , \abk A_{n^2}\}$ be 
a basis of the enveloping algebra 
$\langle N \rangle_{\Q}$, 
where $N=\abk \langle \tau \rangle^H$ and the 
$A_i$ are words in $S$. We can find a finite set 
$\Pi_1$ of primes such that the $\varphi_p(A_i)$ are 
linearly independent and $\varphi_p(1_n-\tau)\neq 0$
for any prime $p\not \in \Pi_1$. 
That is, if $p\not \in \Pi_1$
then $\varphi_p(N)$ is absolutely irreducible and 
contains the transvection $\varphi_p(\tau)$; so
$\varphi_p(H) = \Sp(n, p)$ 
by \cite[Theorem~3.2]{Density}.
Thus $\Pi(H) \subseteq \Pi_1$.
We obtain $\Pi(H)$ after checking whether 
$\varphi_p(H) = \Sp(n, p)$ for each 
$p \in \abk \Pi_1$. 
This last step uses recognition 
algorithms for matrix groups over 
finite fields~\cite{gaprecog}.

\subsection{Integrality and computing the $\Z$-intercept}
\label{integral}

Some of our algorithms require
us to compute the `$\Z$-points' $H_\Z:=H \cap \GL(n, \Z)$
of input $H \leq \GL(n, \Q)$.
 This is possible by the next result. 
\begin{lemma}[{\cite[Lemma~5.1]{Arithm}}]
\label{Integrality}
For a finitely generated subgroup $H$ of $\GL(n, \Q)$,
the following are equivalent:
\begin{itemize}
\item $H$ is integral, i.e., $H$ is 
conjugate to a subgroup of $\GL(n,\Z)$, 
\item $|H:H_\Z|$ is finite,
\item there exists a positive integer $d$ such that $dH$
consists of $\Z$-matrices.
\end{itemize}
\end{lemma}

In \cite[Section~5]{Arithm} we explain
how to find $d$ if $|H : H_{\Z}|$ is finite.
The procedure ${\tt IntegralIntercept}(S,d)$ 
from \cite{ArithSolvable}
then computes a generating set of $H_{\Z}$.  
However, its practicality is limited. 
In our experiments, we calculated a transversal 
of $H_{\Z}$ in $H$  
using an orbit algorithm for the multiplication 
action by $H$, starting with $1_n$. Suppose that
$g$ is an image so obtained.
We test whether $gh^{-1}\in \abk \GL(n,\Z)$ 
for each known orbit element $h$. If this happens for 
some $h$ then $g$ lies in the same coset of $H_{\Z}$ 
(and will yield a Schreier generator of $H_{\Z}$). 
If no such $h$ exists then $g$ is a representative of
 a new coset. 
 
We avail of the following reduction when $|H:H_{\Z}|$ 
is large. Let $\sigma$ be an integer whose prime 
divisors divide the denominators of entries 
in elements of $H$; 
so $\frac{1}{\sigma}\Z\subseteq R$.
Then $H_{\Z}\le K \le H$
where $K = H\cap \Sp(n,\frac{1}{\sigma}\Z)$. 
Membership in $K$ is tested by inspection of 
matrix denominators.
We thus divide the transversal length into
two factors, first calculating a transversal of 
$K$ in $H$, and then a transversal of 
$H_{\Z}$ in $K$. 

A potential complication is too 
many Schreier generators for $H_\Z$. 
Rather than keeping them
all, we randomly select about $300$  
subproducts of Schreier generators for each 
transversal step (cf.~\cite{Babaietal}).
Conceivably we may not then compute all of 
$H_{\Z}$, but merely a proper subgroup. At the 
end we therefore verify, by a calculation in the 
congruence image modulo the level of $H_{\Z}$ 
(see Section~\ref{CSP}), that all Schreier 
generators lie in the subgroup generated 
by the chosen set.

\subsection{The congruence subgroup property and computing}
\label{CSP}
Suppose that $H\leq \abk \Sp(n, \Q)$ is arithmetic, 
i.e., commensurable with $\Sp(n,\Z)$.  
The congruence subgroup property holds for 
$\Sp(n,\Z)$, so that $H_{\Z}$ contains a principal 
congruence subgroup $\ker \varphi_r \cap \Sp(n, \Z)$ for 
some modulus $r>1$. The \emph{level} of $H$, 
denoted $M(H)$, is the modulus 
of the unique maximal principal congruence subgroup  
in $H_{\Z}$. 

Now suppose that $H$ is a finitely generated dense
subgroup of $\Sp(n, \Z)$. The \emph{arithmetic closure} 
$\cl(H)$ of $H$ in $\Sp(n, \Z)$ is the intersection 
of all arithmetic subgroups of $\Sp(n, \Z)$ 
containing $H$ (see \cite[Section~3.3]{Density}).
If $H \leq \Sp(n, \Q)$ is not 
necessarily arithmetic, but $H_{\Z}$ is dense,
then we set $M(H)=M(\cl(H_{\Z}))$.

The level is a key component of our algorithms for computing 
with dense $H\leq \Sp(n,\Z)$.  
If $|\Sp(n, \Z) : \cl(H)|$ is not too large, then we may 
test arithmeticity of $H$ by 
coset enumeration~\cite[Chapter~5]{HEO}. 

The algorithm ${\tt LevelMaxPCS}$ from
\cite{Density} returns $M(H)$ for input 
dense $H\leq \Sp(n, \Z)$ and $\Pi(H)$. 
In practice, for $H \leq \Sp(n, \Q)$ such that 
$H_{\Z}$ is dense, we use
${\tt LevelMaxPCS}(H_{\Z}, \Pi(H_{\Z}))$;
by definition this returns $M(H)$. 
Certainly $\Pi(H) \subseteq \Pi(H_{\Z})$, 
but these sets need not coincide. For example,
$p\in \Pi(H_{\Z})$ could divide $\mu$.

\section{Hypergeometric groups}
\label{Hyper}
\subsection{Background}
\label{4.1}
We adhere mainly to the conventions 
of \cite{BH}.

Let $a = \abk (a_1, \ldots , a_n)$ and 
$b = (b_1, \ldots , b_n)$, 
where $a_j$, $b_k \in \C^\times$ and $a_j\neq b_k$ 
for $1 \leq j,k \leq n$.
A subgroup of $\mathrm{GL}(n, \C)$ generated by 
elements $h_{\infty}$, $h_0$ such that 
$\mathrm{det}(t 1_n - h_{\infty}) 
={\textstyle \prod}_{j=1}^n(t-a_j)$ 
and
$\mathrm{det}(t 1_n - h_0^{-1}) 
= {\textstyle \prod}_{j=1}^n(t-b_j)$
is called a \emph{hypergeometric group}, 
and denoted $H(a,b)$. It is absolutely irreducible by 
\cite[Proposition~3.3]{BH}.
The element $h_1:= (h_0 h_{\infty})^{-1}$ of $H(a,b)$
is a reflection, i.e., $h_1-1_n$ has rank $1$.

If $a_j = \mathrm{exp}(2\pi {\rm i} \alpha_j)$ and  
$b_j = \mathrm{exp}(2\pi {\rm i} \beta_j)$ for 
$\alpha_j, \beta_j \in \C$, 
then $H(a, b)$ is the monodromy group of a hypergeometric 
differential equation~\cite[Proposition~3.2]{BH}. 
\begin{theorem}[{\cite[Theorem~3.5]{BH}}]
\label{BHtheorem}  
For $a_j$, $b_k$ as above, let 
\[
f(t) = {\textstyle \prod}_{j=1}^n(t-a_j) = 
 t^n + A_1t^{n-1}+ \cdots + A_n
\]
and
\[ 
g(t) = {\textstyle \prod}_{j=1}^n(t-b_j) = 
 t^n + B_1t^{n-1}+ \cdots + B_n.
\]
Further, let 
\[
A =  
\left(\begin{array}{cccc} 0 & \cdots & 0 & -A_n
\\
1 & \cdots & 0 & -A_{n-1}
\\
\vdots & \ddots & \vdots & \vdots
\\
0 & \cdots & 1 & -A_{1}
\end{array} \! \right), \quad
B =  
\left(\begin{array}{cccc} 0 & \cdots & 0 & -B_n
\\
1 & \cdots & 0 & -B_{n-1}
\\
\vdots & \ddots & \vdots & \vdots
\\
0 & \cdots & 1 & -B_{1}
\end{array} \! \right).
\]
Then $h_{\infty} = A$, $h_0 = B^{-1}$
generate a hypergeometric group $H(a, b)$ for 
$a = (a_1, \ldots , a_n)$ and $b = (b_1, \ldots , b_n)$. 
Any hypergeometric group with the same $a$, $b$ is
$\mathrm{GL}(n, \C)$-conjugate to this one. 
\end{theorem} 

We are concerned with $H(a, b)$ that are
\begin{itemize}
\item[(i)]  symplectic, 
\item[(ii)] dense in $\Sp_n$, 
\item[(iii)] integral.
\end{itemize}
\noindent There are only finitely many 
$\GL(n,\Q)$-conjugacy classes of such $H(a,b)$.
By \cite[Proposition~6.1]{BH}, $H(a, b)$ is symplectic 
if and only if 
$\{a_1, \ldots , a_n\} = 
 \{a_1^{-1}, \abk \ldots , a_n^{-1}\}$, 
$\{b_1, \ldots , b_n\} = 
 \{b_1^{-1}, \ldots , b_n^{-1}\}$,
and $\delta := \mathrm{det}(h_1) = 1$ 
(whence $h_1$ is a transvection). 
We remark that $H(a,b)$ need not be dense in 
$\Sp_n$ (by, e.g., \cite[Theorem~6.5]{BH}).
Additionally, $H(a, b)$ is integral if and only if 
the $a_j$ and $b_k$ are roots of unity.  
Hence the characteristic polynomials $f(t)$, $g(t)$ 
of $A$, $B$ should be products of 
coprime cyclotomic polynomials.
Under these conditions, $H(a, b) \leq \Sp(\Phi, \Z)$ 
for some $\Phi$. Since $H(a,b)$ is absolutely 
irreducible, $\Phi$ is unique up to a scalar multiple. 

\subsection{Strategy}
\label{4.2}
Assuming that $H(a,b)\leq \GL(n,\Q)$ satisfy the 
requirements (i), (ii), (iii) of 
Section~\ref{4.1}, we proceed as follows.
\begin{itemize}
\item[(I)] 
We list all pairs $f(t)$, $g(t)$ of polynomials of degree $n$, 
each of which is the product of coprime cyclotomic polynomials, 
and such that $\delta = 1$
(for $h_\infty$, $h_0$ as in Theorem~\ref{BHtheorem}). 
Non-dense $H(a,b)$ are excluded by running 
${\tt IsDense}(H(a,b), h_1)$.

\item[(II)] 
The matrix $\Phi$ of a symplectic form 
fixed by $H(a,b)$ may be interpreted as a homomorphism 
between the natural module of $H(a,b)$ and its dual. We 
use MeatAxe techniques~\cite[Section~7.5.2]{HEO} 
to compute $\Phi$. Next, 
$g \in \GL(n, \Q)$ such that $g J_n
g^\top = \Phi$ is found by simple linear algebra. Then
$L=L(a,b) := g^{-1}H(a,b)g \leq \Sp(n, \Q)$.
(We seek a copy of $H(a,b)$ that preserves the standard 
form because it is more convenient for computing; e.g., 
we have a presentation of $\Sp(n,\Z)$ but not of 
$\Sp(\Phi,\Z)$.) Note that $h := \abk
g^{-1}h_1g$ is a transvection in $L$.

Since $\Phi$ is not strictly unique, 
and $g$ can vary by factors stabilizing the form,
$L$ depends on choices made. These might also 
impact $|L:L_{\Z}|$, which we want to keep small for 
efficiency purposes and to avoid large 
increases in the number of Schreier generators 
(remember that $|L:L_{\Z}|<\infty$ is finite by 
Lemma~\ref{Integrality}).
Our code therefore uses heuristics to determine
candidates for $\Phi$ and $g$. It calculates 
$L=\langle T\rangle$ and 
$\overline{k}= \mathrm{lcm} \, \{k\; |\; l^k\in 
L_{\Z} \ \, \forall \ l\in T \}$
(as a stand-in for $|L:L_{\Z}|$). Then $g$ is 
chosen so that $\overline{k}$ is minimal. 

\item[(III)]  
We compute $L_{\Z}$ (see Section~\ref{integral}).
 Although perhaps $h\not \in L_{\Z}$, we can always 
find a transvection 
$\lambda = \abk h^k \in L_{\Z}$ for some $k$.

\item[(IV)] We compute 
\begin{itemize}
\item[] $\Pi(L_{\Z}) =
{\tt PrimesForDense}(L_{\Z}, \lambda)$
\item[] ${\tt LevelMaxPCS}(L_{\Z}, \Pi(L_{\Z}))$
\item[]  $|\Sp(n,\Z):\mathrm{cl}(L_{\Z})|$.
\end{itemize}

\item[(V)] When $|\Sp(n, \Z) : \cl(L_{\Z})|$ is 
sufficiently small, we express the generators of 
$L_{\Z}$ as words in generators of 
$\Sp(n,\Z)$~\cite{Words}, and try to find 
$|\Sp(n, \Z) : L_{\Z}|$ by coset enumeration. 
If this succeeds, i.e., confirms that the indices 
of $L_{\Z}$ and $\mathrm{cl}(L_{\Z})$ are equal,
then we have proved that $L_{\Z}$
and thereby $H(a,b)$ are arithmetic
(cf.~\cite[Theorem~4.1, p.~204]{PlatonovRapinchuk}). 

As the cost of coset enumeration is bounded below by the index, we restricted
our attempts 
to groups with (presumed) indices less than $10^7$.
If the index was expected to be in the range 
$10^7, \ldots , 10^{14}$, then we tried to find an 
intermediate subgroup $L_{\Z}<U<\Sp(n, \Z)$ such 
that $|U : L_{\Z}|\le 10^7$. 
Enumeration was undertaken with a presentation 
for $U$ found by Reidemeister--Schreier 
rewriting~\cite[Chapter~5]{HEO}.
Suitable $U$ are generated by $L_{\Z}$ together
with congruence subgroups in $\Sp(n, \Z)$ of level 
dividing the level of $\mbox{cl}(L_{\Z})$.
\end{itemize}

By \cite[Theorem~1.1]{SinghVenky}, 
if the leading coefficient 
of $f(t)-g(t)$ has absolute value 
at most $2$ then $H(a, b)$ is arithmetic in 
$\Sp(\Phi, \Z)$. 
At least in degree $4$, we proved arithmeticity 
(and computed the level and index) 
whenever the criterion from \cite{SinghVenky} 
applies, and occasionally when it does not. 
Unfortunately, we lack a method for proving 
non-arithmeticity if
coset enumeration fails.

\section{Experimental results}\label{append}

Our algorithms have been implemented in {\sf GAP}~\cite{Gap}.
In this section, we present the complete results of various
experiments for $n=4$ (Table~\ref{nequals4}),
and a sample for $n=\abk 6$ (Table~\ref{nequals6}).
The results for all 916 groups of degree $6$ are 
available at
{\small
\url{https://www.math.colostate.edu/~hulpke/paper/hypergeom6.pdf}}.

\vspace{1.5pt}

A group with Nr $\leq$ 60 in Table~\ref{nequals4}
has the same number in \cite[Table~1]{SinghVenky},  
while Nr $=m\ge$ 100 matches number $m-100$ in
\cite[Table~2]{SinghVenky}. 
The index and arithmeticity questions for these groups were
studied previously in~\cite{HofmannvanStraten,Singh2}.
Table~\ref{tabnumtrans} gives a correspondence between our 
numbering and the notation used in those two papers.
`$S_xy$' denotes the group in Table $x$
of~\cite{Singh2} with line label $y$. 
`$HS(d,k)$' denotes the group labeled
$(d,k)$ in~\cite{HofmannvanStraten}.

\begin{table}[h]
\begin{tabular}{rl||rl||rl}
1&$HS(1,2)$ &118&$S_{1}12(6)$ &136&$S_{3}7$\\
101&$S_{2}1(1)$, $HS(16,8)$ &119&$S_{2}8(2)$ &137&$S_{4}4$\\
102&$S_{1}1$, $HS(9,6)$ &120&$S_{1}7(1)$ &138&$S_{4}5$\\
103&$S_{2}2$, $HS(12,7)$ &121&$S_{1}10(4)$ &139&$S_{3}8$\\
104&$S_{1}2$, $HS(4,4)$ &122&$S_{1}9(3)$ &140&$S_{4}6$\\
105&$S_{2}3$, $HS(8,6)$ &123&$S_{2}12(6)$ &141&$S_{4}7(1)$\\
106&$S_{1}3$, $HS(6,5)$ &124&$S_{2}10(4)$ &142&$S_{3}9(1)$\\
107&$S_{2}4$, $HS(5,5)$ &125&$S_{2}13(7)$ &143&$S_{4}8(2)$\\
108&$S_{2}5$, $HS(4,5)$ &126&$S_{4}1$ &144&$S_{3}10(2)$\\
109&$S_{1}4$, $HS(3,4)$ &127&$S_{3}2$ &145&$S_{4}11(6)$\\
110&$S_{1}5$, $HS(2,3)$ &128&$S_{3}3$ &146&$S_{3}14(7)$\\
111&$S_{2}6$, $HS(2,4)$ &129&$S_{3}4(4)$ &147&$S_{3}12(5)$\\
112&$S_{1}6$, $HS(1,3)$ &130&$S_{3}5$ &148&$S_{4}10(5)$\\
113&$S_{2}7$, $HS(1,4)$ &131&$S_{3}6$ &149&$S_{3}13(6)$\\
114&$S_{2}11(5)$ &132&$S_{4}2$ &150&$S_{4}9(4)$\\
115&$S_{2}9(3)$ &133&$S_{3}1$ &151&$S_{3}15(8)$\\
116&$S_{1}8(2)$ &134&$S_{4}3(3)$\\
117&$S_{1}11(5)$ &135&$S_{3}11(3)$\\
\end{tabular}
\caption{Label correspondences}
\label{tabnumtrans}
\end{table}

The column `Polynomials' in Table~\ref{nequals4} lists
$f(t)$, $g(t)$ as in (I) of 
Section~\ref{4.2} (the Nr entries in Table~\ref{nequals6} 
derive from the listing of these polynomials).
If $\mu >1$ then its prime divisors are given in 
column `Mu'. 
`Int' is $|L:L_{\Z}|$.
`iLevel' and `iIndex' are level and index of 
$\cl (L_\Z)$ in $\Sp(n,\Z)$, respectively. 
`Coeff' is the absolute value of the leading 
coefficient of $f(t)-g(t)$.
The column `Enum' records whether coset 
enumeration succeeded. We reiterate that if an enumeration 
terminates (signified by a tick $\checkmark$) then it returns 
$|\Sp(n,\Z):L_{\Z}|$, and the input group is arithmetic. 
A dash means that coset enumeration failed to terminate;
of course, this 
does not prove that the group is not arithmetic. 
If the group is thin by an argument of \cite{BravThomas}
(see \cite[Table~2]{Singh2}), so that coset enumeration 
is sure to fail, then a letter T is given.
In cases where large $|\Sp(n,\Z):\mathrm{cl}(L_{\Z})|$ 
implies that coset enumeration is unlikely to succeed, 
we put a cross. Indices in degree $4$ were small enough 
to attempt coset enumeration for all groups that are not 
known to be thin. The size of many indices in degree $6$ 
dissuades any attempt at coset enumeration.

We discuss some test groups of interest.
Table~\ref{nequals4} shows that
the groups Nr $=$ 104, 109 
are arithmetic; the question is open 
in \cite[Table~3]{SinghVenky} (rows 5 and 10) but
proved in \cite[Table~1]{Singh2} (rows 2 and 4).

Our method proves arithmeticity for 
all groups in Tables~1 and 3 of \cite{Singh2} (which are 
listed, respectively proved, to be arithmetic there), 
apart from Nr $=$ 102, 120, 122, 128, 135 
(which all have large index, greater than $4\cdot 10^8$).

The arithmetic groups Nr $=$ 126 and 141 
in Table~\ref{nequals4} are No.s~1 and 7(1) in 
\cite[Table~4]{Singh2}, for which arithmeticity 
or thinness was unknown.
Furthermore, our calculations
show that the groups with Nr $=$ 137, 138, 148, 150 
would have rather small index if they were 
arithmetic. Failure of coset enumeration
therefore seems to suggest that, unless we have been 
unlucky, these groups are more likely to be thin than not.

If the polynomial pair includes 
$(t-1)^4$, then a special base change (also featuring
in~\cite{HofmannvanStraten}) was used, to render
the group in $\Sp(4,\Z)$.
This comprises the basis defined in \cite[(2.2)]{doranmorgan},
followed by the two base changes indicated in Remark~1 and 
Equation~(9) of \cite{ChenYangYui}. Thus, it becomes possible 
to compare our results with those of \cite{HofmannvanStraten}.
We see that $|\Sp(n,\Z):L|$ is the same for Nr $=$ 1, 104,
109, 110, 112; i.e., the groups
$(d,k)$ $=$ (1,2), (4,4), (3,4), (2,3), (1,3) 
in \cite[Theorem~4.3]{HofmannvanStraten}. 
Hofmann and van Straten were not able to compute the 
index of $(d,k)=$ (6,5), whereas we could do so for the 
corresponding group Nr $=$ 106, confirming their estimate. 
Although we were unable to prove 
arithmeticity of the group $(d,k)=(9,6)$, Nr $=$ 102, 
the index $2^83^{14}5^2$ that we computed 
is slightly better than the 
estimate in~\cite{HofmannvanStraten}.

Let  $H_n$, $G_n$ be $H(a,b)$ with $f(t) = (t - 1)^n$ 
and $g(t) = (t^{n+1} - 1)/(t - 1)$, $(t^{n+1} + 1)/(t + 1)$,
respectively. 
The arithmeticity problem for $H_n$ and $G_n$ was posed at the 
`Workshop on Thin Groups and Super Approximation', Institute 
for Advanced Study, Princeton, March 2016.
If $n=4$ then $H_n$ is thin~\cite{BravThomas}; see row 
107 in Table~\ref{nequals4}, or row 8 in \cite[Table~3]{Density}
for $\mathrm{cl}(H_4)$.
The group $G_4$ is arithmetic~\cite[Corollary~1.4]{SinghVenky} 
(row 1 in \cite[Table~3]{Density} and row 112 in 
Table~\ref{nequals4}). 

In degree $6$ we proved arithmeticity for a smaller number of
the groups not covered by \cite[Theorem~1.1]{SinghVenky};
two notable exceptions are rows 468 and 534 of 
Table~\ref{nequals6}.
We have not yet solved the arithmeticity 
problem for $G_6$ or $H_6$. However, the level and index of 
their arithmetic closures are stated in rows 774 and 838 of 
Table~\ref{nequals6}.

\subsubsection*{Postscript}
During revisions of our paper, we became aware of the preprint
\cite{Bajpaietal}.  \mbox{Bajpai~et al.} settle arithmeticity of many
symplectic hypergeometric groups in degree $6$, using an adaptation of
\cite{SinghVenky}.  We tried their method in degrees $4$ and $6$, and note
that sometimes it succeeds when our method fails; but also vice versa.
\newpage

{\small

\begin{longtable}{r|l|c|r|r|r|c|c}%
Nr&Polynomials&$\mu$&Int&iLevel&iIndex&Coeff&Enum\\
\hline
\endhead
\rule{0pt}{2.6ex}1&$\begin{array}{l}
t^{4}{-}4t^{3}{+}6t^{2}{-}4t{+}1%
\\
t^{4}{-}2t^{3}{+}3t^{2}{-}2t{+}1%
\end{array}$%
&%
$ 1 $&%
$1$&%
$2$&%
$2{\cdot}5$&%
$2
$&%
\checkmark\\%
\rule{0pt}{4ex}2&$\begin{array}{l}
(t{-}1%
)^{2}(t{+}1%
)^{2}\\
t^{4}{+}2t^{3}{+}3t^{2}{+}2t{+}1%
\end{array}$%
&%
$ 3 $&%
$2^2$&%
$2{\cdot}3^2$&%
$2^73^45^2$&%
$2
$&%
\checkmark\\%
\rule{0pt}{4ex}3&$\begin{array}{l}
(t{-}1%
)^{2}(t{+}1%
)^{2}\\
(t^{2}{+}1%
)(t^{2}{+}t{+}1%
)\end{array}$%
&%
$ 2, 3 $&%
$3^2$&%
$2^43^2$&%
$2^{10}3^35^2$&%
$1
$&%
\checkmark\\%
\rule{0pt}{4ex}4&$\begin{array}{l}
(t{-}1%
)^{2}(t{+}1%
)^{2}\\
t^{4}{+}t^{3}{+}t^{2}{+}t{+}1%
\end{array}$%
&%
$ 5 $&%
$5$&%
$2{\cdot}5^2$&%
$2^53^25{\cdot}13$&%
$1
$&%
\checkmark\\%
\rule{0pt}{4ex}5&$\begin{array}{l}
t^{4}{-}2t^{3}{+}3t^{2}{-}2t{+}1%
\\
(t{-}1%
)^{2}(t{+}1%
)^{2}\end{array}$%
&%
$ 3 $&%
$2^2$&%
$2{\cdot}3^2$&%
$2^73^45^2$&%
$2
$&%
\checkmark\\%
\rule{0pt}{4ex}6&$\begin{array}{l}
(t^{2}{-}t{+}1%
)(t^{2}{+}1%
)\\
(t{-}1%
)^{2}(t{+}1%
)^{2}\end{array}$%
&%
$ 2, 3 $&%
$3^2$&%
$2^43^2$&%
$2^{10}3^35^2$&%
$1
$&%
\checkmark\\%
\rule{0pt}{4ex}7&$\begin{array}{l}
t^{4}{-}t^{3}{+}t^{2}{-}t{+}1%
\\
(t{-}1%
)^{2}(t{+}1%
)^{2}\end{array}$%
&%
$ 5 $&%
$5$&%
$2{\cdot}5^2$&%
$2^43^25{\cdot}13$&%
$1
$&%
\checkmark\\%
\rule{0pt}{4ex}8&$\begin{array}{l}
t^{4}{+}2t^{3}{+}3t^{2}{+}2t{+}1%
\\
t^{4}{+}4t^{3}{+}6t^{2}{+}4t{+}1%
\end{array}$%
&%
$ 1 $&%
$1$&%
$2$&%
$2{\cdot}5$&%
$2
$&%
\checkmark\\%
\rule{0pt}{4ex}9&$\begin{array}{l}
(t{-}1%
)^{2}(t^{2}{+}t{+}1%
)\\
t^{4}{+}2t^{2}{+}1%
\end{array}$%
&%
$ 2 $&%
$2{\cdot}3$&%
$2^4$&%
$2^63^25$&%
$1
$&%
\checkmark\\%
\rule{0pt}{4ex}10&$\begin{array}{l}
(t{-}1%
)^{2}(t^{2}{+}t{+}1%
)\\
t^{4}{+}t^{3}{+}t^{2}{+}t{+}1%
\end{array}$%
&%
$ 2, 5 $&%
$2^33{\cdot}5^2$&%
$2^35^2$&%
$2^83^35^213$&%
$2
$&%
\checkmark\\%
\rule{0pt}{4ex}11&$\begin{array}{l}
t^{4}{-}2t^{3}{+}3t^{2}{-}2t{+}1%
\\
(t{-}1%
)^{2}(t^{2}{+}t{+}1%
)\end{array}$%
&%
$ 2 $&%
$3^2$&%
$2^4$&%
$2^83^25$&%
$1
$&%
\checkmark\\%
\rule{0pt}{4ex}12&$\begin{array}{l}
(t{-}1%
)^{2}(t^{2}{+}t{+}1%
)\\
(t{+}1%
)^{2}(t^{2}{-}t{+}1%
)\end{array}$%
&%
$ 2 $&%
$2$&%
$2^4$&%
$2^{11}3{\cdot}5$&%
$2
$&%
\checkmark\\%
\rule{0pt}{4ex}13&$\begin{array}{l}
(t{-}1%
)^{2}(t^{2}{+}t{+}1%
)\\
(t^{2}{-}t{+}1%
)(t^{2}{+}1%
)\end{array}$%
&%
$ 1 $&%
$1$&%
$2^2$&%
$2^63{\cdot}5$&%
$2
$&%
\checkmark\\%
\rule{0pt}{4ex}14&$\begin{array}{l}
(t{-}1%
)^{2}(t^{2}{+}t{+}1%
)\\
t^{4}{+}1%
\end{array}$%
&%
$ 2 $&%
$2$&%
$2^3$&%
$2^23{\cdot}5$&%
$1
$&%
\checkmark\\%
\rule{0pt}{4ex}15&$\begin{array}{l}
(t{-}1%
)^{2}(t^{2}{+}t{+}1%
)\\
t^{4}{-}t^{3}{+}t^{2}{-}t{+}1%
\end{array}$%
&%
$ 1 $&%
$1$&%
$2$&%
$2{\cdot}3$&%
$1
$&%
\checkmark\\%
\rule{0pt}{4ex}16&$\begin{array}{l}
(t{-}1%
)^{2}(t^{2}{+}t{+}1%
)\\
t^{4}{-}t^{2}{+}1%
\end{array}$%
&%
$ 2 $&%
$3$&%
$2^3$&%
$2^33{\cdot}5$&%
$1
$&%
\checkmark\\%
\rule{0pt}{4ex}17&$\begin{array}{l}
t^{4}{+}2t^{2}{+}1%
\\
t^{4}{+}2t^{3}{+}3t^{2}{+}2t{+}1%
\end{array}$%
&%
$ 1 $&%
$1$&%
$2$&%
$2{\cdot}5$&%
$2
$&%
\checkmark\\%
\rule{0pt}{4ex}18&$\begin{array}{l}
(t{+}1%
)^{2}(t^{2}{+}1%
)\\
t^{4}{+}2t^{3}{+}3t^{2}{+}2t{+}1%
\end{array}$%
&%
$ 1 $&%
$1$&%
$2$&%
$2{\cdot}5$&%
$1
$&%
\checkmark\\%
\rule{0pt}{4ex}19&$\begin{array}{l}
t^{4}{+}t^{3}{+}t^{2}{+}t{+}1%
\\
t^{4}{+}2t^{3}{+}3t^{2}{+}2t{+}1%
\end{array}$%
&%
$ 1 $&%
$1$&%
$1$&%
$1$&%
$1
$&%
\checkmark\\%
\rule{0pt}{4ex}20&$\begin{array}{l}
(t{+}1%
)^{2}(t^{2}{-}t{+}1%
)\\
t^{4}{+}2t^{3}{+}3t^{2}{+}2t{+}1%
\end{array}$%
&%
$ 2 $&%
$3^2$&%
$2^4$&%
$2^73^25$&%
$1
$&%
\checkmark\\%
\rule{0pt}{4ex}21&$\begin{array}{l}
t^{4}{+}1%
\\
t^{4}{+}2t^{3}{+}3t^{2}{+}2t{+}1%
\end{array}$%
&%
$ 1 $&%
$1$&%
$2$&%
$2{\cdot}5$&%
$2
$&%
\checkmark\\%
\rule{0pt}{4ex}22&$\begin{array}{l}
t^{4}{-}t^{2}{+}1%
\\
t^{4}{+}2t^{3}{+}3t^{2}{+}2t{+}1%
\end{array}$%
&%
$ 1 $&%
$1$&%
$2^2$&%
$2^63{\cdot}5$&%
$2
$&%
\checkmark\\%
\rule{0pt}{4ex}23&$\begin{array}{l}
t^{4}{+}t^{3}{+}t^{2}{+}t{+}1%
\\
(t{+}1%
)^{2}(t^{2}{+}t{+}1%
)\end{array}$%
&%
$ 1 $&%
$1$&%
$2$&%
$2{\cdot}3$&%
$2
$&%
\checkmark\\%
\rule{0pt}{4ex}24&$\begin{array}{l}
(t{-}1%
)^{2}(t^{2}{+}1%
)\\
t^{4}{-}2t^{3}{+}3t^{2}{-}2t{+}1%
\end{array}$%
&%
$ 1 $&%
$1$&%
$2$&%
$2{\cdot}5$&%
$1
$&%
\checkmark\\%
\rule{0pt}{4ex}25&$\begin{array}{l}
(t{-}1%
)^{2}(t^{2}{+}1%
)\\
(t^{2}{-}t{+}1%
)(t^{2}{+}t{+}1%
)\end{array}$%
&%
$ 3 $&%
$3$&%
$2{\cdot}3^2$&%
$2^53{\cdot}5^2$&%
$2
$&%
\checkmark\\%
\rule{0pt}{4ex}26&$\begin{array}{l}
(t{-}1%
)^{2}(t^{2}{+}1%
)\\
t^{4}{+}1%
\end{array}$%
&%
$ 1 $&%
$1$&%
$2^2$&%
$2^73^25$&%
$2
$&%
\checkmark\\%
\rule{0pt}{4ex}27&$\begin{array}{l}
(t{-}1%
)^{2}(t^{2}{+}1%
)\\
t^{4}{-}t^{3}{+}t^{2}{-}t{+}1%
\end{array}$%
&%
$ 1 $&%
$1$&%
$2$&%
$2{\cdot}3$&%
$1
$&%
\checkmark\\%
\rule{0pt}{4ex}28&$\begin{array}{l}
(t{-}1%
)^{2}(t^{2}{+}1%
)\\
t^{4}{-}t^{2}{+}1%
\end{array}$%
&%
$ 3 $&%
$2$&%
$2{\cdot}3^2$&%
$2^53{\cdot}5^2$&%
$2
$&%
\checkmark\\%
\rule{0pt}{4ex}29&$\begin{array}{l}
t^{4}{+}2t^{2}{+}1%
\\
t^{4}{+}t^{3}{+}t^{2}{+}t{+}1%
\end{array}$%
&%
$ 1 $&%
$1$&%
$2$&%
$2{\cdot}3$&%
$1
$&%
\checkmark\\%
\rule{0pt}{4ex}30&$\begin{array}{l}
t^{4}{-}2t^{3}{+}3t^{2}{-}2t{+}1%
\\
t^{4}{+}2t^{2}{+}1%
\end{array}$%
&%
$ 1 $&%
$1$&%
$2$&%
$2{\cdot}5$&%
$2
$&%
\checkmark\\%
\rule{0pt}{4ex}31&$\begin{array}{l}
t^{4}{+}2t^{2}{+}1%
\\
(t{+}1%
)^{2}(t^{2}{-}t{+}1%
)\end{array}$%
&%
$ 2 $&%
$2{\cdot}3$&%
$2^4$&%
$2^63^25$&%
$1
$&%
\checkmark\\%
\rule{0pt}{4ex}32&$\begin{array}{l}
t^{4}{-}t^{3}{+}t^{2}{-}t{+}1%
\\
t^{4}{+}2t^{2}{+}1%
\end{array}$%
&%
$ 1 $&%
$1$&%
$2$&%
$2{\cdot}3$&%
$1
$&%
\checkmark\\%
\rule{0pt}{4ex}33&$\begin{array}{l}
t^{4}{+}t^{3}{+}t^{2}{+}t{+}1%
\\
(t{+}1%
)^{2}(t^{2}{+}1%
)\end{array}$%
&%
$ 1 $&%
$1$&%
$2$&%
$2{\cdot}3$&%
$1
$&%
\checkmark\\%
\rule{0pt}{4ex}34&$\begin{array}{l}
(t^{2}{-}t{+}1%
)(t^{2}{+}t{+}1%
)\\
(t{+}1%
)^{2}(t^{2}{+}1%
)\end{array}$%
&%
$ 3 $&%
$3$&%
$2{\cdot}3^2$&%
$2^53{\cdot}5^2$&%
$2
$&%
\checkmark\\%
\rule{0pt}{4ex}35&$\begin{array}{l}
t^{4}{+}1%
\\
(t{+}1%
)^{2}(t^{2}{+}1%
)\end{array}$%
&%
$ 1 $&%
$1$&%
$2^2$&%
$2^73^25$&%
$2
$&%
\checkmark\\%
\rule{0pt}{4ex}36&$\begin{array}{l}
t^{4}{-}t^{2}{+}1%
\\
(t{+}1%
)^{2}(t^{2}{+}1%
)\end{array}$%
&%
$ 3 $&%
$2$&%
$2{\cdot}3^2$&%
$2^53{\cdot}5^2$&%
$2
$&%
\checkmark\\%
\rule{0pt}{4ex}37&$\begin{array}{l}
t^{4}{+}t^{3}{+}t^{2}{+}t{+}1%
\\
(t^{2}{+}1%
)(t^{2}{+}t{+}1%
)\end{array}$%
&%
$ 1 $&%
$1$&%
$2$&%
$2{\cdot}3$&%
$1
$&%
\checkmark\\%
\rule{0pt}{4ex}38&$\begin{array}{l}
(t{+}1%
)^{2}(t^{2}{-}t{+}1%
)\\
(t^{2}{+}1%
)(t^{2}{+}t{+}1%
)\end{array}$%
&%
$ 1 $&%
$1$&%
$2^2$&%
$2^63{\cdot}5$&%
$2
$&%
\checkmark\\%
\rule{0pt}{4ex}39&$\begin{array}{l}
t^{4}{+}1%
\\
(t^{2}{+}1%
)(t^{2}{+}t{+}1%
)\end{array}$%
&%
$ 1 $&%
$1$&%
$2^3$&%
$2^23{\cdot}5$&%
$1
$&%
\checkmark\\%
\rule{0pt}{4ex}40&$\begin{array}{l}
t^{4}{-}t^{3}{+}t^{2}{-}t{+}1%
\\
(t^{2}{+}1%
)(t^{2}{+}t{+}1%
)\end{array}$%
&%
$ 1 $&%
$1$&%
$2$&%
$2{\cdot}3$&%
$2
$&%
\checkmark\\%
\rule{0pt}{4ex}41&$\begin{array}{l}
t^{4}{-}t^{2}{+}1%
\\
(t^{2}{+}1%
)(t^{2}{+}t{+}1%
)\end{array}$%
&%
$ 2, 3 $&%
$2{\cdot}3$&%
$2^33^2$&%
$2^63^25^2$&%
$1
$&%
\checkmark\\%
\rule{0pt}{4ex}42&$\begin{array}{l}
(t{+}1%
)^{2}(t^{2}{-}t{+}1%
)\\
t^{4}{+}t^{3}{+}t^{2}{+}t{+}1%
\end{array}$%
&%
$ 1 $&%
$1$&%
$2$&%
$2{\cdot}3$&%
$1
$&%
\checkmark\\%
\rule{0pt}{4ex}43&$\begin{array}{l}
(t^{2}{-}t{+}1%
)(t^{2}{+}t{+}1%
)\\
t^{4}{+}t^{3}{+}t^{2}{+}t{+}1%
\end{array}$%
&%
$ 1 $&%
$1$&%
$1$&%
$1$&%
$1
$&%
\checkmark\\%
\rule{0pt}{4ex}44&$\begin{array}{l}
(t^{2}{-}t{+}1%
)(t^{2}{+}1%
)\\
t^{4}{+}t^{3}{+}t^{2}{+}t{+}1%
\end{array}$%
&%
$ 1 $&%
$1$&%
$2$&%
$2{\cdot}3$&%
$2
$&%
\checkmark\\%
\rule{0pt}{4ex}45&$\begin{array}{l}
t^{4}{+}1%
\\
t^{4}{+}t^{3}{+}t^{2}{+}t{+}1%
\end{array}$%
&%
$ 1 $&%
$1$&%
$2$&%
$2{\cdot}3$&%
$1
$&%
\checkmark\\%
\rule{0pt}{4ex}46&$\begin{array}{l}
t^{4}{-}t^{3}{+}t^{2}{-}t{+}1%
\\
t^{4}{+}t^{3}{+}t^{2}{+}t{+}1%
\end{array}$%
&%
$ 1 $&%
$1$&%
$2^2$&%
$2^93^2$&%
$2
$&%
\checkmark\\%
\rule{0pt}{4ex}47&$\begin{array}{l}
t^{4}{-}t^{2}{+}1%
\\
t^{4}{+}t^{3}{+}t^{2}{+}t{+}1%
\end{array}$%
&%
$ 1 $&%
$1$&%
$1$&%
$1$&%
$1
$&%
\checkmark\\%
\rule{0pt}{4ex}48&$\begin{array}{l}
(t{-}1%
)^{2}(t^{2}{-}t{+}1%
)\\
t^{4}{-}t^{3}{+}t^{2}{-}t{+}1%
\end{array}$%
&%
$ 1 $&%
$1$&%
$2$&%
$2{\cdot}3$&%
$2
$&%
\checkmark\\%
\rule{0pt}{4ex}49&$\begin{array}{l}
t^{4}{-}2t^{3}{+}3t^{2}{-}2t{+}1%
\\
t^{4}{+}1%
\end{array}$%
&%
$ 1 $&%
$1$&%
$2$&%
$2{\cdot}5$&%
$2
$&%
\checkmark\\%
\rule{0pt}{4ex}50&$\begin{array}{l}
t^{4}{-}2t^{3}{+}3t^{2}{-}2t{+}1%
\\
t^{4}{-}t^{3}{+}t^{2}{-}t{+}1%
\end{array}$%
&%
$ 1 $&%
$1$&%
$1$&%
$1$&%
$1
$&%
\checkmark\\%
\rule{0pt}{4ex}51&$\begin{array}{l}
t^{4}{-}2t^{3}{+}3t^{2}{-}2t{+}1%
\\
t^{4}{-}t^{2}{+}1%
\end{array}$%
&%
$ 1 $&%
$1$&%
$2^2$&%
$2^63{\cdot}5$&%
$2
$&%
\checkmark\\%
\rule{0pt}{4ex}52&$\begin{array}{l}
t^{4}{+}1%
\\
(t{+}1%
)^{2}(t^{2}{-}t{+}1%
)\end{array}$%
&%
$ 2 $&%
$2$&%
$2^3$&%
$2^23{\cdot}5$&%
$1
$&%
\checkmark\\%
\rule{0pt}{4ex}53&$\begin{array}{l}
t^{4}{-}t^{3}{+}t^{2}{-}t{+}1%
\\
(t{+}1%
)^{2}(t^{2}{-}t{+}1%
)\end{array}$%
&%
$ 5 $&%
$5$&%
$2{\cdot}5^2$&%
$2^43^25{\cdot}13$&%
$2
$&%
\checkmark\\%
\rule{0pt}{4ex}54&$\begin{array}{l}
t^{4}{-}t^{2}{+}1%
\\
(t{+}1%
)^{2}(t^{2}{-}t{+}1%
)\end{array}$%
&%
$ 2 $&%
$3$&%
$2^3$&%
$2^33{\cdot}5$&%
$1
$&%
\checkmark\\%
\rule{0pt}{4ex}55&$\begin{array}{l}
t^{4}{-}t^{3}{+}t^{2}{-}t{+}1%
\\
(t^{2}{-}t{+}1%
)(t^{2}{+}t{+}1%
)\end{array}$%
&%
$ 1 $&%
$1$&%
$1$&%
$1$&%
$1
$&%
\checkmark\\%
\rule{0pt}{4ex}56&$\begin{array}{l}
(t^{2}{-}t{+}1%
)(t^{2}{+}1%
)\\
t^{4}{+}1%
\end{array}$%
&%
$ 1 $&%
$1$&%
$2^3$&%
$2^23{\cdot}5$&%
$1
$&%
\checkmark\\%
\rule{0pt}{4ex}57&$\begin{array}{l}
t^{4}{-}t^{3}{+}t^{2}{-}t{+}1%
\\
(t^{2}{-}t{+}1%
)(t^{2}{+}1%
)\end{array}$%
&%
$ 1 $&%
$1$&%
$2$&%
$2{\cdot}3$&%
$1
$&%
\checkmark\\%
\rule{0pt}{4ex}58&$\begin{array}{l}
(t^{2}{-}t{+}1%
)(t^{2}{+}1%
)\\
t^{4}{-}t^{2}{+}1%
\end{array}$%
&%
$ 2, 3 $&%
$2{\cdot}3$&%
$2^33^2$&%
$2^63^25^2$&%
$1
$&%
\checkmark\\%
\rule{0pt}{4ex}59&$\begin{array}{l}
t^{4}{-}t^{3}{+}t^{2}{-}t{+}1%
\\
t^{4}{+}1%
\end{array}$%
&%
$ 1 $&%
$1$&%
$2$&%
$2{\cdot}3$&%
$1
$&%
\checkmark\\%
\rule{0pt}{4ex}60&$\begin{array}{l}
t^{4}{-}t^{3}{+}t^{2}{-}t{+}1%
\\
t^{4}{-}t^{2}{+}1%
\end{array}$%
&%
$ 1 $&%
$1$&%
$1$&%
$1$&%
$1
$&%
\checkmark\\%
\rule{0pt}{4ex}101&$\begin{array}{l}
t^{4}{-}4t^{3}{+}6t^{2}{-}4t{+}1%
\\
t^{4}{+}4t^{3}{+}6t^{2}{+}4t{+}1%
\end{array}$%
&%
$ 1 $&%
$1$&%
$2^{10}$&%
$2^{40}3^25$&%
$8
$&%
T\\%
\rule{0pt}{4ex}102&$\begin{array}{l}
t^{4}{-}4t^{3}{+}6t^{2}{-}4t{+}1%
\\
t^{4}{+}2t^{3}{+}3t^{2}{+}2t{+}1%
\end{array}$%
&%
$ 1 $&%
$1$&%
$2{\cdot}3^5$&%
$2^83^{14}5^2$&%
$6
$&%
---\\%
\rule{0pt}{4ex}103&$\begin{array}{l}
t^{4}{-}4t^{3}{+}6t^{2}{-}4t{+}1%
\\
(t{+}1%
)^{2}(t^{2}{+}t{+}1%
)\end{array}$%
&%
$ 1 $&%
$1$&%
$2^53^2$&%
$2^{17}3^65^2$&%
$7
$&%
T\\%
\rule{0pt}{4ex}104&$\begin{array}{l}
t^{4}{-}4t^{3}{+}6t^{2}{-}4t{+}1%
\\
t^{4}{+}2t^{2}{+}1%
\end{array}$%
&%
$ 1 $&%
$1$&%
$2^6$&%
$2^{20}3^25$&%
$4
$&%
\checkmark\\%
\rule{0pt}{4ex}105&$\begin{array}{l}
t^{4}{-}4t^{3}{+}6t^{2}{-}4t{+}1%
\\
(t{+}1%
)^{2}(t^{2}{+}1%
)\end{array}$%
&%
$ 1 $&%
$1$&%
$2^7$&%
$2^{24}3^25$&%
$6
$&%
T\\%
\rule{0pt}{4ex}106&$\begin{array}{l}
t^{4}{-}4t^{3}{+}6t^{2}{-}4t{+}1%
\\
(t^{2}{+}1%
)(t^{2}{+}t{+}1%
)\end{array}$%
&%
$ 1 $&%
$1$&%
$2^33^2$&%
$2^{10}3^65^2$&%
$5
$&%
\checkmark\\%
\rule{0pt}{4ex}107&$\begin{array}{l}
t^{4}{-}4t^{3}{+}6t^{2}{-}4t{+}1%
\\
t^{4}{+}t^{3}{+}t^{2}{+}t{+}1%
\end{array}$%
&%
$ 1 $&%
$1$&%
$2{\cdot}5^3$&%
$2^83^35^813$&%
$5
$&%
T\\%
\rule{0pt}{4ex}108&$\begin{array}{l}
t^{4}{-}4t^{3}{+}6t^{2}{-}4t{+}1%
\\
(t{+}1%
)^{2}(t^{2}{-}t{+}1%
)\end{array}$%
&%
$ 1 $&%
$1$&%
$2^5$&%
$2^{13}3{\cdot}5$&%
$5
$&%
T\\%
\rule{0pt}{4ex}109&$\begin{array}{l}
t^{4}{-}4t^{3}{+}6t^{2}{-}4t{+}1%
\\
(t^{2}{-}t{+}1%
)(t^{2}{+}t{+}1%
)\end{array}$%
&%
$ 1 $&%
$1$&%
$2^23^2$&%
$2^93^55^2$&%
$4
$&%
\checkmark\\%
\rule{0pt}{4ex}110&$\begin{array}{l}
t^{4}{-}4t^{3}{+}6t^{2}{-}4t{+}1%
\\
(t^{2}{-}t{+}1%
)(t^{2}{+}1%
)\end{array}$%
&%
$ 1 $&%
$1$&%
$2^3$&%
$2^63{\cdot}5$&%
$3
$&%
\checkmark\\%
\rule{0pt}{4ex}111&$\begin{array}{l}
t^{4}{-}4t^{3}{+}6t^{2}{-}4t{+}1%
\\
t^{4}{+}1%
\end{array}$%
&%
$ 1 $&%
$1$&%
$2^4$&%
$2^{11}3^25$&%
$4
$&%
T\\%
\rule{0pt}{4ex}112&$\begin{array}{l}
t^{4}{-}4t^{3}{+}6t^{2}{-}4t{+}1%
\\
t^{4}{-}t^{3}{+}t^{2}{-}t{+}1%
\end{array}$%
&%
$ 1 $&%
$1$&%
$2$&%
$2{\cdot}3$&%
$3
$&%
\checkmark\\%
\rule{0pt}{4ex}113&$\begin{array}{l}
t^{4}{-}4t^{3}{+}6t^{2}{-}4t{+}1%
\\
t^{4}{-}t^{2}{+}1%
\end{array}$%
&%
$ 1 $&%
$1$&%
$2^2$&%
$2^55$&%
$4
$&%
T\\%
\rule{0pt}{4ex}114&$\begin{array}{l}
(t{-}1%
)^{2}(t^{2}{+}t{+}1%
)\\
t^{4}{+}4t^{3}{+}6t^{2}{+}4t{+}1%
\end{array}$%
&%
$ 1 $&%
$1$&%
$2^5$&%
$2^{13}3{\cdot}5$&%
$5
$&%
T\\%
\rule{0pt}{4ex}115&$\begin{array}{l}
(t{-}1%
)^{2}(t^{2}{+}1%
)\\
t^{4}{+}4t^{3}{+}6t^{2}{+}4t{+}1%
\end{array}$%
&%
$ 2 $&%
$2^3$&%
$2^8$&%
$2^{27}3^25$&%
$6
$&%
T\\%
\rule{0pt}{4ex}116&$\begin{array}{l}
t^{4}{+}2t^{2}{+}1%
\\
t^{4}{+}4t^{3}{+}6t^{2}{+}4t{+}1%
\end{array}$%
&%
$ 1 $&%
$1$&%
$2^4$&%
$2^{20}3^25$&%
$4
$&%
\checkmark\\%
\rule{0pt}{4ex}117&$\begin{array}{l}
(t^{2}{+}1%
)(t^{2}{+}t{+}1%
)\\
t^{4}{+}4t^{3}{+}6t^{2}{+}4t{+}1%
\end{array}$%
&%
$ 2 $&%
$2{\cdot}3$&%
$2^4$&%
$2^73^25$&%
$3
$&%
\checkmark\\%
\rule{0pt}{4ex}118&$\begin{array}{l}
t^{4}{+}t^{3}{+}t^{2}{+}t{+}1%
\\
t^{4}{+}4t^{3}{+}6t^{2}{+}4t{+}1%
\end{array}$%
&%
$ 1 $&%
$1$&%
$2$&%
$2{\cdot}3$&%
$3
$&%
\checkmark\\%
\rule{0pt}{4ex}119&$\begin{array}{l}
(t{-}1%
)^{2}(t^{2}{-}t{+}1%
)\\
t^{4}{+}4t^{3}{+}6t^{2}{+}4t{+}1%
\end{array}$%
&%
$ 3 $&%
$2^2$&%
$2^53^3$&%
$2^{19}3^65^2$&%
$7
$&%
T\\%
\rule{0pt}{4ex}120&$\begin{array}{l}
t^{4}{-}2t^{3}{+}3t^{2}{-}2t{+}1%
\\
t^{4}{+}4t^{3}{+}6t^{2}{+}4t{+}1%
\end{array}$%
&%
$ 1 $&%
$1$&%
$2{\cdot}3^4$&%
$2^83^{14}5^2$&%
$6
$&%
---\\%
\rule{0pt}{4ex}121&$\begin{array}{l}
(t^{2}{-}t{+}1%
)(t^{2}{+}t{+}1%
)\\
t^{4}{+}4t^{3}{+}6t^{2}{+}4t{+}1%
\end{array}$%
&%
$ 3 $&%
$2^23$&%
$2^23^3$&%
$2^{11}3^65^2$&%
$4
$&%
\checkmark\\%
\rule{0pt}{4ex}122&$\begin{array}{l}
(t^{2}{-}t{+}1%
)(t^{2}{+}1%
)\\
t^{4}{+}4t^{3}{+}6t^{2}{+}4t{+}1%
\end{array}$%
&%
$ 2, 3 $&%
$2^33$&%
$2^43^3$&%
$2^{13}3^75^2$&%
$5
$&%
---\\%
\rule{0pt}{4ex}123&$\begin{array}{l}
t^{4}{+}1%
\\
t^{4}{+}4t^{3}{+}6t^{2}{+}4t{+}1%
\end{array}$%
&%
$ 1 $&%
$1$&%
$2^3$&%
$2^{11}3^25$&%
$4
$&%
T\\%
\rule{0pt}{4ex}124&$\begin{array}{l}
t^{4}{-}t^{3}{+}t^{2}{-}t{+}1%
\\
t^{4}{+}4t^{3}{+}6t^{2}{+}4t{+}1%
\end{array}$%
&%
$ 1 $&%
$1$&%
$2{\cdot}5^2$&%
$2^73^35^813$&%
$5
$&%
T\\%
\rule{0pt}{4ex}125&$\begin{array}{l}
t^{4}{-}t^{2}{+}1%
\\
t^{4}{+}4t^{3}{+}6t^{2}{+}4t{+}1%
\end{array}$%
&%
$ 1 $&%
$1$&%
$2^2$&%
$2^55$&%
$4
$&%
T\\%
\rule{0pt}{4ex}126&$\begin{array}{l}
(t{-}1%
)^{2}(t^{2}{+}t{+}1%
)\\
(t{+}1%
)^{2}(t^{2}{+}1%
)\end{array}$%
&%
$ 2 $&%
$2^23$&%
$2^6$&%
$2^{11}3^25$&%
$3
$&%
\checkmark\\%
\rule{0pt}{4ex}127&$\begin{array}{l}
(t{-}1%
)^{2}(t^{2}{+}1%
)\\
t^{4}{+}2t^{3}{+}3t^{2}{+}2t{+}1%
\end{array}$%
&%
$ 3 $&%
$2^2$&%
$2{\cdot}3^2$&%
$2^73^45^2$&%
$4
$&%
\checkmark\\%
\rule{0pt}{4ex}128&$\begin{array}{l}
(t{-}1%
)^{2}(t^{2}{-}t{+}1%
)\\
t^{4}{+}2t^{3}{+}3t^{2}{+}2t{+}1%
\end{array}$%
&%
$ 2, 3 $&%
$2^23^2$&%
$2^43^3$&%
$2^{13}3^75^2$&%
$5
$&%
---\\%
\rule{0pt}{4ex}129&$\begin{array}{l}
t^{4}{-}2t^{3}{+}3t^{2}{-}2t{+}1%
\\
t^{4}{+}2t^{3}{+}3t^{2}{+}2t{+}1%
\end{array}$%
&%
$ 1 $&%
$1$&%
$2^4$&%
$2^{16}3{\cdot}5$&%
$4
$&%
\checkmark\\%
\rule{0pt}{4ex}130&$\begin{array}{l}
(t^{2}{-}t{+}1%
)(t^{2}{+}1%
)\\
t^{4}{+}2t^{3}{+}3t^{2}{+}2t{+}1%
\end{array}$%
&%
$ 2 $&%
$3^2$&%
$2^4$&%
$2^73^25$&%
$3
$&%
\checkmark\\%
\rule{0pt}{4ex}131&$\begin{array}{l}
t^{4}{-}t^{3}{+}t^{2}{-}t{+}1%
\\
t^{4}{+}2t^{3}{+}3t^{2}{+}2t{+}1%
\end{array}$%
&%
$ 1 $&%
$1$&%
$1$&%
$1$&%
$3
$&%
\checkmark\\%
\rule{0pt}{4ex}132&$\begin{array}{l}
(t{-}1%
)^{2}(t^{2}{+}1%
)\\
(t{+}1%
)^{2}(t^{2}{+}t{+}1%
)\end{array}$%
&%
$ 2, 3 $&%
$2{\cdot}3$&%
$2^53^2$&%
$2^{14}3^25^2$&%
$5
$&%
---\\%
\rule{0pt}{4ex}133&$\begin{array}{l}
t^{4}{+}2t^{2}{+}1%
\\
(t{+}1%
)^{2}(t^{2}{+}t{+}1%
)\end{array}$%
&%
$ 2 $&%
$2{\cdot}3$&%
$2^43$&%
$2^73^45^2$&%
$3
$&%
\checkmark\\%
\rule{0pt}{4ex}134&$\begin{array}{l}
(t{-}1%
)^{2}(t^{2}{-}t{+}1%
)\\
(t{+}1%
)^{2}(t^{2}{+}t{+}1%
)\end{array}$%
&%
$ 2 $&%
$2$&%
$2^43^2$&%
$2^{18}3^85^2$&%
$6
$&%
---\\%
\rule{0pt}{4ex}135&$\begin{array}{l}
t^{4}{-}2t^{3}{+}3t^{2}{-}2t{+}1%
\\
(t{+}1%
)^{2}(t^{2}{+}t{+}1%
)\end{array}$%
&%
$ 2, 3 $&%
$2^23^2$&%
$2^43^2$&%
$2^{13}3^75^2$&%
$5
$&%
---\\%
\rule{0pt}{4ex}136&$\begin{array}{l}
(t{+}1%
)^{2}(t^{2}{+}t{+}1%
)\\
(t^{2}{-}t{+}1%
)(t^{2}{+}1%
)\end{array}$%
&%
$ 3 $&%
$3$&%
$2^23^2$&%
$2^{10}3^35^2$&%
$4
$&%
\checkmark\\%
\rule{0pt}{4ex}137&$\begin{array}{l}
(t{+}1%
)^{2}(t^{2}{+}t{+}1%
)\\
t^{4}{+}1%
\end{array}$%
&%
$ 1 $&%
$1$&%
$2^23$&%
$2^23^35^2$&%
$3
$&%
---\\%
\rule{0pt}{4ex}138&$\begin{array}{l}
t^{4}{-}t^{3}{+}t^{2}{-}t{+}1%
\\
(t{+}1%
)^{2}(t^{2}{+}t{+}1%
)\end{array}$%
&%
$ 5 $&%
$5$&%
$2{\cdot}5^2$&%
$2^43^25{\cdot}13$&%
$4
$&%
---\\%
\rule{0pt}{4ex}139&$\begin{array}{l}
t^{4}{-}t^{2}{+}1%
\\
(t{+}1%
)^{2}(t^{2}{+}t{+}1%
)\end{array}$%
&%
$ 2 $&%
$3$&%
$2^23$&%
$2^33^35^2$&%
$3
$&%
\checkmark\\%
\rule{0pt}{4ex}140&$\begin{array}{l}
(t{-}1%
)^{2}(t^{2}{+}1%
)\\
t^{4}{+}t^{3}{+}t^{2}{+}t{+}1%
\end{array}$%
&%
$ 2, 5 $&%
$2^33{\cdot}5^2$&%
$2^35^2$&%
$2^83^35^213$&%
$3
$&%
---\\%
\rule{0pt}{4ex}141&$\begin{array}{l}
(t{-}1%
)^{2}(t^{2}{+}1%
)\\
(t{+}1%
)^{2}(t^{2}{-}t{+}1%
)\end{array}$%
&%
$ 2 $&%
$3$&%
$2^3$&%
$2^93^25$&%
$3
$&%
\checkmark\\%
\rule{0pt}{4ex}142&$\begin{array}{l}
(t{-}1%
)^{2}(t^{2}{-}t{+}1%
)\\
t^{4}{+}2t^{2}{+}1%
\end{array}$%
&%
$ 2 $&%
$2{\cdot}3$&%
$2^43$&%
$2^73^45^2$&%
$3
$&%
\checkmark\\%
\rule{0pt}{4ex}143&$\begin{array}{l}
(t{-}1%
)^{2}(t^{2}{-}t{+}1%
)\\
(t{+}1%
)^{2}(t^{2}{+}1%
)\end{array}$%
&%
$ 3 $&%
$3$&%
$2^33^2$&%
$2^{13}3^25^2$&%
$5
$&%
---\\%
\rule{0pt}{4ex}144&$\begin{array}{l}
t^{4}{-}2t^{3}{+}3t^{2}{-}2t{+}1%
\\
(t{+}1%
)^{2}(t^{2}{+}1%
)\end{array}$%
&%
$ 3 $&%
$2^2$&%
$2{\cdot}3^2$&%
$2^73^45^2$&%
$4
$&%
\checkmark\\%
\rule{0pt}{4ex}145&$\begin{array}{l}
t^{4}{-}t^{3}{+}t^{2}{-}t{+}1%
\\
(t{+}1%
)^{2}(t^{2}{+}1%
)\end{array}$%
&%
$ 2, 5 $&%
$2^33{\cdot}5$&%
$2^25$&%
$2^73^35{\cdot}13$&%
$3
$&%
---\\%
\rule{0pt}{4ex}146&$\begin{array}{l}
(t{-}1%
)^{2}(t^{2}{-}t{+}1%
)\\
(t^{2}{+}1%
)(t^{2}{+}t{+}1%
)\end{array}$%
&%
$ 3 $&%
$3$&%
$2^23^2$&%
$2^{10}3^35^2$&%
$4
$&%
\checkmark\\%
\rule{0pt}{4ex}147&$\begin{array}{l}
t^{4}{-}2t^{3}{+}3t^{2}{-}2t{+}1%
\\
(t^{2}{+}1%
)(t^{2}{+}t{+}1%
)\end{array}$%
&%
$ 2 $&%
$3^2$&%
$2^4$&%
$2^83^25$&%
$3
$&%
\checkmark\\%
\rule{0pt}{4ex}148&$\begin{array}{l}
(t{-}1%
)^{2}(t^{2}{-}t{+}1%
)\\
t^{4}{+}t^{3}{+}t^{2}{+}t{+}1%
\end{array}$%
&%
$ 5 $&%
$5$&%
$2{\cdot}5^2$&%
$2^53^25{\cdot}13$&%
$4
$&%
---\\%
\rule{0pt}{4ex}149&$\begin{array}{l}
t^{4}{-}2t^{3}{+}3t^{2}{-}2t{+}1%
\\
t^{4}{+}t^{3}{+}t^{2}{+}t{+}1%
\end{array}$%
&%
$ 1 $&%
$1$&%
$1$&%
$1$&%
$3
$&%
\checkmark\\%
\rule{0pt}{4ex}150&$\begin{array}{l}
(t{-}1%
)^{2}(t^{2}{-}t{+}1%
)\\
t^{4}{+}1%
\end{array}$%
&%
$ 2 $&%
$2$&%
$2^33$&%
$2^33^35^2$&%
$3
$&%
---\\%
\rule{0pt}{4ex}151&$\begin{array}{l}
(t{-}1%
)^{2}(t^{2}{-}t{+}1%
)\\
t^{4}{-}t^{2}{+}1%
\end{array}$%
&%
$ 2 $&%
$3$&%
$2^33$&%
$2^33^35^2$&%
$3
$&%
\checkmark\\%
\caption{Degree $4$}
\label{nequals4}
\end{longtable}

}

\begin{landscape}{\small 
\begin{longtable}{r|l|c|r|r|r|c|c}%
Nr&Polynomials& Mu &Int&iLevel&iIndex&Coeff&Enum\\
\hline
\endhead
\rule{0pt}{4ex}158&$\begin{array}{l}
(t^{2}{-}t{+}1%
)(t^{2}{+}1%
)^{2}\\
(t^{2}{+}t{+}1%
)(t^{4}{-}t^{3}{+}t^{2}{-}t{+}1%
)\end{array}$%
&%
$ 2 $&%
$3$&%
$2^3$&%
$2^43^37$&%
$1
$&%
\checkmark\\%
\rule{0pt}{4ex}162&$\begin{array}{l}
(t{-}1%
)^{2}(t{+}1%
)^{2}(t^{2}{+}1%
)\\
t^{6}{-}t^{3}{+}1%
\end{array}$%
&%
$ 3 $&%
$3$&%
$2{\cdot}3^2$&%
$2^43{\cdot}7^213$&%
$1
$&%
---\\%
\rule{0pt}{4ex}167&$\begin{array}{l}
(t{-}1%
)^{2}(t{+}1%
)^{2}(t^{2}{+}1%
)\\
(t^{2}{+}t{+}1%
)(t^{4}{-}t^{2}{+}1%
)\end{array}$%
&%
$ 3 $&%
$2^2$&%
$2{\cdot}3$&%
$2^93^35{\cdot}7^213$&%
$1
$&%
$\times$\\%
\rule{0pt}{4ex}390&$\begin{array}{l}
t^{6}{-}t^{5}{+}t^{4}{-}t^{3}{+}t^{2}{-}t{+}1%
\\
(t{+}1%
)^{2}(t^{4}{-}t^{3}{+}t^{2}{-}t{+}1%
)\end{array}$%
&%
$ 2, 7 $&%
$2^63{\cdot}5{\cdot}7^2$&%
$2^37^2$&%
$2^{14}3^55{\cdot}7^219{\cdot}43$&%
$2
$&%
$\times$\\%
\rule{0pt}{4ex}394&$\begin{array}{l}
(t{-}1%
)^{2}(t^{2}{-}t{+}1%
)(t^{2}{+}1%
)\\
t^{6}{-}t^{5}{+}t^{4}{-}t^{3}{+}t^{2}{-}t{+}1%
\end{array}$%
&%
$ 2 $&%
$2^33{\cdot}5{\cdot}7$&%
$2^3$&%
$2^83^35{\cdot}7$&%
$2
$&%
\checkmark\\%
\rule{0pt}{4ex}437&$\begin{array}{l}
t^{6}{+}t^{5}{+}t^{4}{+}t^{3}{+}t^{2}{+}t{+}1%
\\
(t{+}1%
)^{2}(t^{2}{+}t{+}1%
)^{2}\end{array}$%
&%
$ 1 $&%
$1$&%
$2$&%
$2^53^2$&%
$3
$&%
---\\%
\rule{0pt}{4ex}468&$\begin{array}{l}
t^{6}{-}t^{5}{+}t^{4}{-}t^{3}{+}t^{2}{-}t{+}1%
\\
(t^{2}{+}t{+}1%
)(t^{4}{+}t^{3}{+}t^{2}{+}t{+}1%
)\end{array}$%
&%
$ 1 $&%
$1$&%
$1$&%
$1$&%
$3
$&%
\checkmark\\%
\rule{0pt}{4ex}534&$\begin{array}{l}
(t^{2}{-}t{+}1%
)(t^{4}{-}t^{3}{+}t^{2}{-}t{+}1%
)\\
t^{6}{+}t^{5}{+}t^{4}{+}t^{3}{+}t^{2}{+}t{+}1%
\end{array}$%
&%
$ 1 $&%
$1$&%
$1$&%
$1$&%
$3
$&%
\checkmark\\%
\rule{0pt}{4ex}774&$\begin{array}{l}
t^{6}{-}6t^{5}{+}15t^{4}{-}20t^{3}{+}15t^{2}{-}6t{+}1%
\\
t^{6}{-}t^{5}{+}t^{4}{-}t^{3}{+}t^{2}{-}t{+}1%
\end{array}$%
&%
$ 1 $&%
$1$&%
$2$&%
$2^23^2$&%
$5
$&%
---\\%
\rule{0pt}{4ex}819&$\begin{array}{l}
t^{6}{-}6t^{5}{+}15t^{4}{-}20t^{3}{+}15t^{2}{-}6t{+}1%
\\
t^{6}{-}t^{3}{+}1%
\end{array}$%
&%
$ 1 $&%
$1$&%
$2{\cdot}3$&%
$2^33^55{\cdot}7^213$&%
$6
$&%
---\\%
\rule{0pt}{4ex}838&$\begin{array}{l}
t^{6}{-}6t^{5}{+}15t^{4}{-}20t^{3}{+}15t^{2}{-}6t{+}1%
\\
t^{6}{+}t^{5}{+}t^{4}{+}t^{3}{+}t^{2}{+}t{+}1%
\end{array}$%
&%
$ 1 $&%
$1$&%
$2{\cdot}7^2$&%
$2^{15}3^65^27^{22}19{\cdot}43$&%
$7
$&%
$\times$\\%
\caption{Degree $6$}
\label{nequals6}
\end{longtable}}
\end{landscape}

\subsection*{Acknowledgments}
We thank Professor Martin Kassabov for his helpful 
advice. Many thanks are due also
to our referees, whose comments led to improvements 
of the paper.
We are grateful to Mathematisches 
Forschungsinstitut Oberwolfach and
the Hausdorff Research Institute for Mathematics
for hospitality and facilitation of our work during 
visits in 2018.
A.~S.~Detinko was supported by Marie Sk\l odowska-Curie 
Individual Fellowship grant H2020 MSCA-IF-2015, 
no.~704910 (EU Framework Programme for Research and 
Innovation). A.~Hulpke was supported by National 
Science Foundation grant DMS-1720146.

\bibliographystyle{amsplain}

\begin{thebibliography}{10}

\bibitem{Babaietal}
L.~Babai, E.~Luks, and \'{A}.~Seress,
\textit{Fast management of permutation groups. I}, 
SIAM J. Comput. \textbf{26} (1997), no.~5, 1310--1342.

\bibitem{Bajpaietal}
J.~Bajpai, S.~Singh, and S.~V.~Singh, \emph{Symplectic 
hypergeometric groups of degree six},
\url{https://arxiv.org/pdf/2003.10191.pdf}

\bibitem{BH}
F.~Beukers and G.~Heckman, 
\textit{Monodromy for the hypergeometric 
function $_nF_{n-1}$}, Invent. Math. \textbf{95} (1989), 
no.~2, 325--354.

\bibitem{BravThomas}
C.~Brav and H.~Thomas, \textit{Thin monodromy in $\Sp(4)$}, 
Compos. Math. \textbf{150} (2014), no.~3, 333--343.

\bibitem{ChenYangYui}
Y.-H.~Chen, Y.~Yang, and N.~Yui,
\textit{Monodromy of Picard--Fuchs differential equations for Calabi--Yau
threefolds},
J. Reine Angew. Math. \textbf{616} (2008), 167--203.

\bibitem{ArithSolvable}
A.~S.~Detinko, D.~L.~Flannery, and W. de Graaf,
\textit{Integrality and arithmeticity of solvable linear groups}, 
J. Symbolic Comput. \textbf{68} (2015), part~1, 138--145.

\bibitem{Arithm}
A.~S. Detinko, D.~L. Flannery, and A.~ Hulpke, 
\textit{Algorithms for arithmetic groups with the congruence 
subgroup property},
J. Algebra \textbf{421} (2015), 234--259. 

\bibitem{Density}
A.~S. Detinko, D.~L. Flannery, and A.~ Hulpke, 
\textit{Zariski density and computing in arithmetic groups}, 
Math. Comp. \textbf{87} (2018), no.~310, 967--986. 

\bibitem{DensityFurther}
A.~S.~Detinko, D.~L.~Flannery, and A.~Hulpke, 
\textit{Algorithms for experimenting with Zariski dense 
subgroups}, Exp. Math., published online June 2018
{\small \url{https://doi.org/10.1080/10586458.2018.1466217}}.

\bibitem{SAT}
A.~S. Detinko, D.~L. Flannery, and A.~ Hulpke, 
\textit{The strong approximation theorem and computing 
with linear groups}, J. Algebra \textbf{529} (2019), 536--549.

\bibitem{doranmorgan}
C.~F.~Doran and J.~W.~Morgan,
\textit{Mirror symmetry and integral variations of Hodge structure
underlying one-parameter families of Calabi--Yau threefolds},
in \textit{Mirror symmetry. V}, 517--537, 
AMS/IP Stud. Adv. Math., \textbf{38}, Amer. Math. Soc., Providence, RI, 2006. 

\bibitem{FuchsMonodromy}
E.~Fuchs, C.~Meiri, and P.~Sarnak, 
\textit{Hyperbolic monodromy groups for 
the hypergeometric equation and Cartan involutions}, 
J. Eur. Math. Soc. (JEMS) \textbf{16} (2014), no.~8, 1617--1671. 

\bibitem{Gap}
The GAP~Group, \emph{GAP -- Groups, Algorithms, and Programming}
 {\small \url{https://www.gap-system.org}}.

\bibitem{HofmannvanStraten} 
J.~Hofmann and D.~van Straten,  \textit{Some 
monodromy groups of finite index in $\mathrm{Sp}_4(\Z)$},
J.~Aust.~Math.~Soc. \textbf{99} (2015), no.~1, 48--62.

\bibitem{HEO}
D.~F.~Holt, B.~Eick, and E.~A.~O'Brien,
\textit{Handbook of computational group theory},
Chapman \& Hall/CRC, Boca Raton, FL, 2005.

\bibitem{Words}
A.~Hulpke, \emph{Constructive membership tests in 
some infinite matrix groups},
 {P}roceedings of the 2018 {ACM} {I}nternational {S}ymposium 
on {S}ymbolic and {A}lgebraic {C}omputation, ACM, New York, 2018, 
 pp.~215--222.

\bibitem{LubotzkySegal}
A.~Lubotzky and D.~Segal, \textit{Subgroup growth}, 
Birkh\"{a}user Verlag, Basel, 2003.

\bibitem{gaprecog}
M.~Neunh\"{o}ffer, \'{A}.~Seress, et al., 
The {\sf GAP} package  {\tt recog},
\emph{A collection of group recognition methods},
{\small \url{http://gap-packages.github.io/recog/}}.

\bibitem{PlatonovRapinchuk}
V.~Platonov and A.~Rapinchuk,
\textit{Algebraic groups and number theory},
 Academic Press, Inc.,  Boston, MA, 1994.

\bibitem{Sarnak} P.~Sarnak, \emph{Notes on thin matrix groups}, 
Thin groups and superstrong approximation, Math. Sci. Res. Inst. Publ., 
vol.~61, Cambridge Univ. Press,
  Cambridge, 2014, pp.~343--362.

\bibitem{Singh}
S.~Singh, \textit{Arithmeticity of four hypergeometric monodromy
groups associated to Calabi--Yau threefolds}, 
Int. Math. Res. Not. IMRN 2015, no.~18, 8874--8889.

\bibitem{Singh2} 
S.~Singh,  \textit{Arithmeticity of some hypergeometric 
monodromy groups in $\mathrm{Sp}(4)$},
J.~Algebra \textbf{473} (2017), 142--165.
      
\bibitem{SinghVenky}
S.~Singh and T.~N.~Venkataramana, \textit{Arithmeticity of certain
symplectic hypergeometric groups}, Duke Math. J. \textbf{163}
(2014), no.~3, 591--617.

\end{thebibliography}

\end{document}